\documentclass[reqno]{amsart}
\usepackage{amsmath}
\usepackage{amsthm}
\usepackage{amssymb}
\usepackage{epsfig}
\usepackage[all,dvips]{xy}

\newcommand{\Z}{\mathbb{Z}}
\renewcommand{\b}{\mathbb{B}}
\newcommand{\W}{\mathbb{W}}
\renewcommand{\H}{\mathbf{H}}
\newcommand{\h}[1]{\mbox{$\mathbf{H}$\boldmath $#1$}}
\newcommand{\G}{\mathcal{G}}

\newcommand{\bn}{\ensuremath{\beta}}
\newcommand{\e}{\ensuremath{\varepsilon}}

\theoremstyle{plain}
\newtheorem{thm}{Theorem}[section]

\newtheorem{cor}[thm]{Corollary}
\newtheorem{lem}[thm]{Lemma}
\newtheorem{prop}[thm]{Proposition}
\theoremstyle{definition}
\newtheorem{defn}[thm]{Definition}
\newtheorem*{induct}{Inductive Hypothesis}
\newtheorem{question}{Question}
\newtheorem{example}[thm]{Example}

\theoremstyle{remark}
\newtheorem{rem}[thm]{Remark}

\title{Homotopy Type of the Boolean Complex of a Coxeter System}

\author{K\'{a}ri Ragnarsson}
\address{The Mathematical Sciences Research Institute, Berkeley, California}
\email{karirag@msri.org}

\author{Bridget Eileen Tenner}
\address{Department of Mathematical Sciences, DePaul University, Chicago, Illinois}
\email{bridget@math.depaul.edu}


\begin{document}

\begin{abstract}
In any Coxeter group, the set of elements whose principal order ideals are boolean forms a simplicial poset under the Bruhat order.  This simplicial poset defines a cell complex, called the boolean complex.  In this paper it is shown that, for any Coxeter system of rank $n$, the boolean complex is homotopy equivalent to a wedge of $(n-1)$-dimensional spheres. The number of such spheres can be computed recursively from the unlabeled Coxeter graph, and defines a new graph invariant called the boolean number.  Specific calculations of the boolean number are given for all finite and affine irreducible Coxeter systems, as well as for systems with graphs that are disconnected, complete, or stars.  One implication of these results is that the boolean complex is contractible if and only if a generator of the Coxeter system is in the center of the group.\\

\end{abstract}

\maketitle

\section{Introduction}

The boolean complex of a finitely generated Coxeter system $(W,S)$ arises from the Bruhat order on $W$. Regarding $W$ as a poset in the Bruhat order, we define the \emph{boolean ideal} $\b(W,S) \subseteq W$ to be the subposet consisting of those elements whose principal (lower) order ideals are boolean. The boolean ideal is a simplicial poset, and, as the name suggests, it is an order ideal of $W$.  In fact, it is maximal among order ideals that are simplicial posets. The \emph{boolean complex} is defined as the regular cell complex $\Delta(W,S)$ whose face poset is the (simplicial) poset $\b(W,S)$. 

The elements in $\b(W,S)$ are easily described: an element in $\b(W,S)$ is an element of $W$ that can be written as a product of distinct elements from the generating set $S$.  Consequently the boolean complex is pure, with each maximal cell having dimension $|S|-1$. These elements play an important role in the study of Coxeter groups because their boolean nature has a variety of consequences related to $R$-polynomials, Kazhdan-Lusztig polynomials, and $g$-polynomials (see \cite{brenti}).

As described above, elements of the boolean ideal are products of distinct elements of the generating set $S$, and thus are governed by the commutativity of elements of $S$.  Consequently, this ideal is determined by the Coxeter graph of $(W,S)$. Recall that the Coxeter graph $G = G(W,S)$ has vertex set $S$, with an edge between vertices $s$ and $t$ if and only if $s$ and $t$ do not commute in $W$. An edge is labeled by the order $m(s,t)$ of $st \in W$ when $m(s,t) > 3$.  Since we are only concerned with the commutativity of generators in this paper, we suppress the labels and consider the underlying unlabeled graph.  Because elements of $S$ are involutions, the elements $s$ and $t$ commute if and only if $m(s,t) =2$, and hence if and only if the vertices $s$ and $t$ are non-adjacent in $G$.  For more information about Coxeter systems, see \cite{bjornerbrenti}.

From the graph $G$, one constructs a simplicial poset $\b(G)$ whose elements are equivalence classes of strings of distinct elements of $S$, where two strings are equivalent if one can be transformed into the other by commuting elements that are non-adjacent in $G$. The partial order on $\b(G)$ is induced by substring inclusion. Of course, when $G = G(W,S)$, the poset $\b(G)$ is isomorphic to the poset $\b(W,S)$, so this construction recovers the boolean ideal from the (unlabeled) Coxeter graph. We refer to $\b(G)$ as the \emph{boolean ideal of $G$}, and to the associated regular cell complex $\Delta(G)$ as the \emph{boolean complex of $G$}.

If $G$ is the complete graph, then $\b(G)$ is the complex of injective words, which has previously been studied by Farmer \cite{farmer}, Bj\"{o}rner and Wachs \cite{bjornerwachs}, and Reiner and Webb \cite{reiner-webb}.  The complete graph is treated in Corollary~\ref{cor:complete}.

\begin{example}\label{ex:a2}
If the graph $G$ consists of two vertices and a single edge between them, then the poset $\b(G)$ and the boolean complex $\Delta(G)$ are depicted in Figure~\ref{fig:ex1}.
\begin{figure}[htbp]
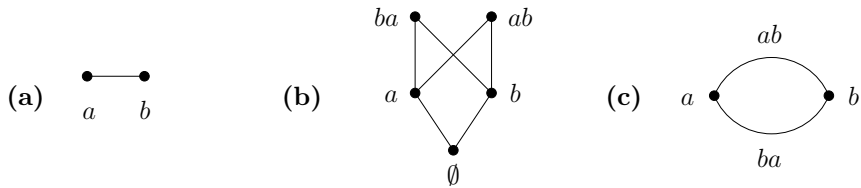

\parbox[c]{.35in}{\mbox{\bf{(a)}}} \parbox[c]{1in}{\scalebox{.4}{\input{ex1-graph.pstex_t}}}
\parbox[c]{.35in}{\mbox{\bf{(b)}}} \parbox[c]{1.25in}{\scalebox{.4}{\input{ex1-poset.pstex_t}}}
\parbox[c]{.35in}{\mbox{\bf{(c)}}} \parbox[c]{1.25in}{\scalebox{.4}{\input{ex1-complex.pstex_t}}}
\caption{\textbf{(a)} A graph $G$.  \textbf{(b)} The poset $\b(G)$.  \textbf{(c)} The boolean complex $\Delta(G)$, with geometric realization $|\Delta(G)| \simeq S^1$.}\label{fig:ex1}
\end{figure}
\end{example}

Because their unlabeled Coxeter graphs are the same, Example~\ref{ex:a2} applies to the Coxeter groups $A_2, B_2/C_2, G_2$ and $I_2(m)$, and shows that in each case the geometric realization of the boolean complex is homotopy equivalent to the unit circle $S^1$.  In this paper, we prove more generally that the boolean complex of any finite simple graph, and hence of any Coxeter system, has geometric realization homotopy equivalent to a wedge of top dimensional spheres, and give a recursive formula for calculating the number of these spheres.  In specific cases, including the finite and affine irreducible Coxeter systems, we calculate this number explicitly.

The subsequent organization of this paper begins with a section precisely defining the primary objects and putting this project in the greater context of the study of the Bruhat order of Coxeter group.  Section~\ref{section:main results} states the main result of the article, that the boolean complex for any finite simple graph has geometric realization homotopy equivalent to a wedge of a particular number of top dimensional spheres.  The homotopy types for the boolean complexes of the finite and affine irreducible Coxeter systems are given as a corollary in this section.  Section~\ref{section:morse} discusses discrete Morse theory, which is the main tool in the proof presented in Section~\ref{section:proof}.  A selection of corollaries to the main theorem are given in Section~\ref{section:corollaries}. Section \ref{sec:Homology} contains a discussion on how generating cycles for the homology of Boolean complex enumerate the spheres occurring in the wedge sum representing the homotopy type of its geometric realization. The paper concludes with suggestions for follow-up questions in Section~\ref{section:follow-up}.

\section{Motivation and definitions}\label{section:motivation}

As noted above, the motivation for this work is the study of Coxeter systems and Coxeter graphs, and the importance of their boolean elements.  Henceforth, all Coxeter systems are assumed to be finitely generated.

We use standard poset terminology throughout this paper, and refer the reader to \cite{ec1} for more background.

\begin{defn}
Let $W$ be a group defined on generators $S$.  The pair $(W,S)$ is a \emph{Coxeter system} if the relations in $W$ are of the form $s^2 = id$ for all $s \in S$, and $(st)^{m(s,t)} = id$ for $s,t \in S$ and $m(s,t) \in \mathbb{Z}^+ \cup \{ \infty \}$.
\end{defn}

Because $S$ consists of involutions, two elements $s,t \in S$ commute if and only if $m(s,t) =2$. The involution condition also implies that $m(s,t) = m(t,s)$. Since $S$ generates $W$, any $\sigma \in W$ can be written as a word on letters in $S$. That is, $\sigma$ admits an expression $\sigma = s_1 \cdots s_\ell$, where $s_i \in S$.  The product $s_1 \cdots s_\ell$ is a \emph{reduced expression} for $\sigma$ if it is of minimal length $\ell$, in which case $\ell$ is the \emph{length} of $\sigma$.

\begin{defn}
For a Coxeter system $(W,S)$, the \emph{(strong) Bruhat order} is the partial order on $W$ where $\sigma \leq \tau$ if and only if $\sigma$ admits a reduced expression that is a subword of a reduced expression of $\tau$.
\end{defn}

\begin{rem}\label{rem:rank -1}
The Bruhat order makes $W$ a ranked poset, with rank determined by length.  Because the minimal element in a simplicial poset corresponds to the empty face in the geometric realization of that poset, we make the the convention that this minimal element has rank $-1$, thus emphasizing that the face data in the poset is contained in the non-negative ranks.
\end{rem}

The structure of the Bruhat order for finite Coxeter groups was studied by the second author in \cite{tenner}.  One aspect of this study was a description of elements with boolean principal order ideals.

\begin{defn}
Let $(W,S)$ be a Coxeter system and regard $W$ as a poset under the Bruhat order. An element $\sigma \in W$ is \emph{boolean} if its principal (lower) order ideal in $W$ is isomorphic to a boolean algebra.  The \emph{boolean ideal} $\b(W,S) \subseteq W$ is the subposet of boolean elements.  
\end{defn}

It is clear from the definition that $\b(W,S)$ is an order ideal in the Bruhat order, thus justifying the terminology. It is also clear, by construction, that $\b(W,S)$ is a simplicial poset.

Because the boolean poset is simplicial, the interval between any two comparable elements is boolean and, in particular, a lattice.  By \cite{brenti}, this has several implications for the Kazhdan-Lusztig polynomials, $R$-polynomials, and $g$-polynomials.  For example, for any $\sigma, \tau \in \b(W,S)$ with $\sigma \le \tau$, we have
\begin{equation*}
P_{\sigma, \tau}(q) = g([\sigma, \tau]^*;q)  =1,
\end{equation*}
\noindent where $P_{\sigma, \tau}(q)$ is the Kazhdan-Lusztig polynomial, and
\begin{equation*}
R_{\sigma, \tau}(q) = \widetilde{G}_{[\sigma, \tau]}(q) = (q-1)^{\ell(\tau) - \ell(\sigma)},
\end{equation*}
\noindent where $\ell$ is the length function.  Further structural and computational consequences can be found in \cite{brenti}.

The results in \cite{tenner} state that for the finite Coxeter groups of types $A$, $B$, and $D$, boolean elements can be characterized by pattern avoidance.  Moreover, the boolean elements of these groups are enumerated by length.  For example, the number of boolean elements of length $k$ in the finite Coxeter group $A_n$ is
\begin{equation}\label{eqn:boolean of length k in A_n}
\sum_{i=1}^k\binom{n+1-i}{k+1-i}\binom{k-1}{i-1}.
\end{equation}
As the unlabeled Coxeter graphs for $A_n$ and $B_n$ are identical, the boolean elements of length $k$ in $B_n$ are also enumerated by \eqref{eqn:boolean of length k in A_n}.  For the group $D_n$, the enumeration is more complicated, and a recursive formula is given in \cite{tenner}.

The following lemma is immediate from the description of the Bruhat order above, and gives a useful characterization of boolean elements.

\begin{lem} \label{lem:BooleanElements}
Let $(W,S)$ be a Coxeter system.  An element of $W$ is boolean if and only if it has no repeated letters in its reduced expressions.
\end{lem}

It follows from the lemma that every maximal element in $\b(W,S)$ has the same rank, equal to $|S|-1$.  

\begin{defn} 
Let $(W,S)$ be a Coxeter system. The \emph{boolean complex} of $(W,S)$ is the regular cell complex $\Delta(W,S)$ whose face poset is the simplicial poset $\b(W,S)$.
\end{defn}

The existence of such a complex $\Delta(W,S)$ follows from a well-known result about simplicial posets, and in fact about CW-posets (see \cite{bjorner}).

The minimal element of $\b(W,S)$ represents the empty cell, and an element of rank $k \geq 0$ represents a $k$-dimensional cell (see Remark~\ref{rem:rank -1}). One can think of the cells in $\b(W,S)$ as simplices, because the minimal subcomplex containing each cell is isomorphic to a simplex of the same dimension. Nevertheless, the boolean complex itself is not a simplicial complex because the cells are not determined by the vertices they contain; for instance, there are two 1-cells, $12$ and $21$, with the same vertices in Example \ref{ex:a2}.  

One obtains a geometric realization $|\Delta(W,S)|$ of the boolean complex in the standard way, by taking one geometric simplex of dimension $k$ for each cell of dimension $k$, and gluing them together according to the face poset.  The \emph{homotopy type} of a complex is understood to mean the homotopy type of its geometric realization.

The main result of this paper, Theorem~\ref{thm:main result}, shows that $|\Delta(W,S)|$ has the homotopy type of a wedge of spheres of dimension $|S|-1$. Moreover, we give a recursive formula for computing the number of spheres in the wedge. To describe this recursion, we present an alternative construction of the boolean complex, in terms of the unlabeled Coxeter graph of $(W,S)$.

\begin{defn}
The \emph{Coxeter graph} $G(W,S)$ of a Coxeter system $(W,S)$ has vertex set $S$ and an edge between $s$ and $t$ if and only if $m(s,t) > 2$.  An edge corresponding to $m(s,t) > 3$ is labeled by $m(s,t)$. The \emph{unlabeled Coxeter graph} is the underlying simple graph obtained by omitting all edge labels.
\end{defn}

A Coxeter system $(W,S)$ can be recovered from its Coxeter graph $G$. Taking the vertex set of $G$ to be $S$, one forms the group $W$ generated by $S$ subject to the relations mandated by the edges in $G$, and the condition that $S$ should consist of involutions. An unlabeled Coxeter graph, however, contains less information, and only allows one to determine when two elements in $S$ commute. Thus, if one is only concerned with commutativity of generators, this graph suffices.

\begin{example}
Let $W$ be the group generated by $S = \{a,b,c,d\}$, with relations $s^2 = id$ for all $s \in S$, and $m(a,c) = m(a,d) = m(c,d) =2,m(a,b) = 3$, $m(b,c) = m(b,d) = 4$.  The (unlabeled) Coxeter graph $G(W,S)$ is shown below.
\begin{equation*}
\scalebox{.4}{\begin{picture}(0,0)%
\epsfig{file=ex-coxgraph.pstex}%
\end{picture}%
\setlength{\unitlength}{3947sp}%
\begingroup\makeatletter\ifx\SetFigFont\undefined%
\gdef\SetFigFont#1#2#3#4#5{%
  \reset@font\fontsize{#1}{#2pt}%
  \fontfamily{#3}\fontseries{#4}\fontshape{#5}%
  \selectfont}%
\fi\endgroup%
\begin{picture}(3607,2196)(4118,-5344)
\put(7126,-5236){\makebox(0,0)[lb]{\smash{{\SetFigFont{25}{30.0}{\familydefault}{\mddefault}{\updefault}$d$}}}}
\put(4126,-3886){\makebox(0,0)[lb]{\smash{{\SetFigFont{25}{30.0}{\familydefault}{\mddefault}{\updefault}$a$}}}}
\put(5626,-3886){\makebox(0,0)[lb]{\smash{{\SetFigFont{25}{30.0}{\familydefault}{\mddefault}{\updefault}$b$}}}}
\put(7126,-3436){\makebox(0,0)[lb]{\smash{{\SetFigFont{25}{30.0}{\familydefault}{\mddefault}{\updefault}$c$}}}}
\end{picture}%
}
\end{equation*}
\end{example}

For a finite simple graph $G$ with vertex set $S$, define the poset $\b(G)$ as follows. First, let $\W(S)$ be the set of words on $S$ with no repeated letters, ordered by the subword order relation. A typical element in $\W(S)$ is thus of the form $s_1 s_2 \cdots s_\ell$, where $s_1, s_2, \ldots, s_\ell$ are distinct elements of $S$. Next, consider the equivalence relation generated by the requirement that
\[ s_1 \cdots s_{i-1} s_i s_{i+1} s_{i+2} \cdots s_\ell \sim s_1 \cdots s_{i-1} s_{i+1} s_{i} s_{i+2} \cdots s_\ell \]
if $\{s_i,s_{i+1} \}$ is not an edge in $G$.  Let $\b(G)$ be the set of equivalence classes of $\W(S)$ with respect to this equivalence relation. A preimage of an element $\sigma \in \b(G)$ is called a \emph{word representative}. Note that the set of letters occurring in each word representative of $\sigma$ is the same. We say that $\sigma$ \emph{contains} a letter $s$ if $s$ occurs in the string representatives of $\sigma$. A partial ordering is induced on the set $\b(G)$ from the subword order. That is, $\sigma < \tau$ in $\b(G)$ if some word representative of $\sigma$ is a subword of a word representative of $\tau$.

\begin{defn}
The \emph{boolean ideal} of a finite simple graph $G$ is the poset $\b(G)$.
\end{defn}

The motivation for the construction of $\b(G)$ is of course the following obvious fact, which we record as a lemma.

\begin{lem}
If $G$ is the unlabeled Coxeter graph of a Coxeter system $(W,S)$, then $\b(G) \cong \b(W,S)$.
\end{lem}

It can be shown directly, or via the relationship between graphs and Coxeter systems, that $\b(G)$ is a simplicial poset for all finite simple graphs $G$.

\begin{defn}
The \emph{boolean complex} of a finite simple graph $G$ is the regular cell complex $\Delta(G)$ associated to $\b(G)$.
\end{defn}

The main result of this paper, as mentioned above, can be equivalently stated as saying that for any finite simple graph $G$ with vertex set $S$, the geometric realization $|\Delta(G)|$ is homotopy equivalent to a wedge of $(|S|-1)$-dimensional spheres.  It is this version of the result that we shall prove. The promised recursive formula for the number of spheres is given in terms of basic graph operations. These results are stated precisely in Theorem \ref{thm:main result}.

The Euler characteristic $\chi$ of a regular cell complex $\Delta$, and likewise the Euler characteristic of its geometric realization $|\Delta|$, is the alternating sum of the number of faces $f_r$ of each rank $r \ge 0$ in $\Delta$:
\begin{equation*}
\chi(\Delta) = \sum_{r \ge 0} (-1)^r f_r.
\end{equation*}
\noindent Given Remark~\ref{rem:rank -1}, this can also be computed by enumerating each non-negative rank in the corresponding simplicial poset.
In particular, the enumeration from \cite{tenner} cited in \eqref{eqn:boolean of length k in A_n} enables the calculation of the Euler characteristic of the boolean complex for the Coxeter group $A_n$.

\begin{cor}\label{cor:euler characteristic}
For all $n \ge 1$,
\begin{equation*}
\chi(\Delta(A_n)) = (-1)^{n-1}f(n-1) + 1,
\end{equation*}
\noindent where $\{f(n)\}$ are the Fibonacci numbers.
\end{cor}

Note the relationship between $\{\chi(\Delta(A_n))\}$ and sequences A008346 and A119282 in \cite {oeis}: sequence A008346 is equal to $\{|\chi(\Delta(A_n))|\}$, while sequence A119282 is equal to $\{-\chi(\Delta(A_n))\}$.  Also, Corollary~\ref{cor:euler characteristic} foreshadows the fact that $|\Delta(A_n)|$ is homotopy equivalent to the wedge of $f(n-1)$ $(n-1)$-spheres.

Before stating the main results precisely, it is informative to mention similar work which has been done for the \emph{independence complex}.  Ehrenborg and Hetyei \cite{ehrenborg} and Kozlov \cite{kozlov} prove, each in the context of different results and frameworks, that the complex of sparse sets of $[n]$ is contractible in some cases and homotopy equivalent to a sphere in the remaining cases.  In the context of the Bruhat order, the sparse subsets of $[n]$ correspond to the fully commutative elements in the Coxeter group $A_n$.  That is, all letters in such a reduced expression commute with each other.  Thus, these results show that the complex formed from the subposet of $\b(A_n)$ consisting of the fully commutable elements is either contractible or homotopy equivalent to a sphere.

\section{Main results}\label{section:main results}

In this section, we state the main result of the article and draw consequences for the classical Coxeter groups.  The proof of the main theorem is rather technical, and is postponed until Section~\ref{section:proof}.  

It is convenient to use the notation
 \[ b \cdot S^r :=  \bigvee_{i=1}^b S^r \]
for a wedge sum of $b$ spheres of dimension $r$. Since the wedge sum is the coproduct in the category of pointed spaces, $0 \cdot S^r$ then denotes a single point. 

Graph-theoretic notation will also be used in the statement of the theorem and in its proof.

\begin{defn}
For a finite graph $G$, let $|G|$ denote the number of vertices in $G$.
\end{defn}

\begin{defn}
Let $G$ be a finite simple graph and $e$ an edge in $G$.
\begin{itemize}
\item \emph{Deletion:} $G-e$ is the graph obtained by deleting the edge $e$.
\item \emph{Simple contraction:} $G/e$ is the graph obtained by contracting the edge $e$ and then removing all loops and redundant edges.
\item \emph{Extraction:} $G-[e]$ is the graph obtained by removing the edge $e$ and its incident vertices.
\end{itemize}
\end{defn}

\begin{defn}
For $n \ge 1$, let $\delta_n$ be the graph consisting of $n$ disconnected vertices.
\end{defn}

We can now state the main theorem of this paper.  We use the symbol $\simeq$ to denote homotopy equivalence.

\begin{thm}\label{thm:main result}
For every nonempty, finite simple graph $G$, there is an integer $\bn(G)$ so that
\begin{equation*}
|\Delta(G)| \simeq \bn(G) \cdot S^{|G|-1}.
\end{equation*}
Moreover, the values $\bn(G)$ can be computed using the recursive formula
\begin{equation} \label{eqn:m + edge}
  \bn(G) = \bn(G - e) + \bn(G/e) + \bn(G -[e]),
\end{equation}
if $e$ is an edge in $G$ such that $G -[e]$ is nonempty, with initial conditions
$$\bn(A_2) = 1\text{\ and\ } \bn(\delta_n) = 0,$$
where $A_2$ is the graph with two vertices and one edge.
\end{thm}

The integer $\bn(G)$ will be called the \emph{boolean number} of the graph $G$. Notice that by formally
setting $\bn(\emptyset) = 1$, the boolean number can be extended so that the recursive formula in equation~\eqref{eqn:m + edge} holds for $A_2$.

As discussed in Section \ref{section:motivation}, the above theorem implies that the geometric realization of the boolean complex of a Coxeter system $(W,S)$ is homotopy equivalent to a wedge of spheres of dimension $|S|-1$. The number of spheres occurring in the wedge can be calculated recursively using equation~\eqref{eqn:m + edge}. This process can be greatly expedited by the following proposition, which shows that the boolean number is multiplicative with respect to connected components.

\begin{prop}\label{prop:disjoint union}
If $G = H_1 \sqcup H_2$ for graphs $H_1$ and $H_2$, then
\[ \Delta(G) = \Delta(H_1) * \Delta(H_2), \]
where $*$ denotes simplicial join, and consequently
\[ |\Delta(G)| \simeq \bn(H_1)\bn(H_2) \cdot S^{|H_1| + |H_2| -1}. \]
In particular, $\bn(G) = \bn(H_1) \bn(H_2)$.
\end{prop}
\begin{proof}
Since $H_1$ and $H_2$ are disjoint, every element of $\b(H_1)$ commutes with every element of $\b(H_2)$.  Thus, the complex $\Delta(G)$ is formed by taking the simplicial join of the complexes $\Delta(H_1)$ and $\Delta(H_2)$, and hence the geometric realization $|\Delta(G)|$ is the topological join of $|\Delta(H_1)|$ and $|\Delta(H_2)|$ (see \cite{hatcher}). The last claim now follows from Theorem~\ref{thm:main result}.
\end{proof}

The homotopy types of the boolean complexes associated to the finite and affine irreducible Coxeter systems can be calculated as a corollary to Theorem \ref{thm:main result}.

\begin{cor}
The homotopy types of the boolean complexes for the finite and affine irreducible Coxeter systems are listed below, where $f(n)$ is the sequence of Fibonacci numbers (with $f(1) = f(2) =1$) and $c(n)$ is the sequence defined by $c(1) = 1$ and $c(n) = c(n-1) + f(n) + f(n-2)$.\\
\begin{center}
$\begin{array}{c|l}
\text{Coxeter} & \text{Homotopy type}\\
\text{group $W$} & \text{of $|\Delta(W,S)|$}\\
\hline
\rule[0mm]{0mm}{4mm}
A_n & f(n-1) \cdot S^{n-1}\\
\rule[0mm]{0mm}{4mm}
B_n & f(n-1) \cdot S^{n-1}\\
\rule[0mm]{0mm}{4mm}
D_n & f(n-2) \cdot S^{n-1}\\
\rule[0mm]{0mm}{4mm}
E_6 & 4 \cdot S^5\\
\rule[0mm]{0mm}{4mm}
E_7 & 6 \cdot S^6\\
\rule[0mm]{0mm}{4mm}
E_8 & 10 \cdot S^7\\
\rule[0mm]{0mm}{4mm}
F_4 & 2 \cdot S^3\\
\rule[0mm]{0mm}{4mm}
G_2 & S^1\\
\rule[0mm]{0mm}{4mm}
H_3 & S^2\\
\rule[0mm]{0mm}{4mm}
H_4 & 2 \cdot S^3\\
\rule[0mm]{0mm}{4mm}
I_2(m) & S^1
\end{array}$
\hspace*{.2in}
$\begin{array}{c|l}
\text{Coxeter} & \text{Homotopy type}\\
\text{group $W$} & \text{of $|\Delta(W,S)|$}\\
\hline
\rule[0mm]{0mm}{4mm}
\widetilde{A}_n & c(n) \cdot S^n\\
\rule[0mm]{0mm}{4mm}
\widetilde{B}_n & f(n-2) \cdot S^n\\
\rule[0mm]{0mm}{4mm}
\widetilde{C}_n & f(n-1) \cdot S^n\\
\rule[0mm]{0mm}{4mm}
\widetilde{D}_n & f(n-3) \cdot S^n\\
\rule[0mm]{0mm}{4mm}
\widetilde{E}_6 & 7 \cdot S^6\\
\rule[0mm]{0mm}{4mm}
\widetilde{E}_7 & 9 \cdot S^7\\
\rule[0mm]{0mm}{4mm}
\widetilde{E}_8 & 16 \cdot S^8\\
\rule[0mm]{0mm}{4mm}
\widetilde{F}_4 & 3 \cdot S^4 \\
\rule[0mm]{0mm}{4mm}
\widetilde{G}_2 & S^2\\
\multicolumn{2}{c}{\rule[0mm]{0mm}{4mm}\ }\\
\multicolumn{2}{c}{\rule[0mm]{0mm}{4mm}\ }
\end{array}$
\end{center}
\ \\
\end{cor}

The sequence $c(n)$ is entry A014739 in \cite{oeis}.  It can be written in closed form as
\begin{equation*}
c_n = \left(\frac{1+\sqrt{5}}{2}\right)^{n+1}+\left(\frac{1-\sqrt{5}}{2}\right)^{n+1}-2.
\end{equation*}

\section{Discrete Morse theory} \label{section:morse}

The primary tool in the proof of Theorem \ref{thm:main result} is discrete Morse theory, which gives an expedient way to analyze the homotopy type of the geometric realization of a regular cell complex through combinatorial properties of its face poset. Discrete Morse theory is a rich subject, and in this section we present only those ideas and results necessary for our argument. The reader is encouraged to read \cite{forman-complexes, forman-user} for a detailed background.

\begin{defn}
Let $P$ be a ranked poset. A \emph{matching} on $P$ is a collection $M$ of pairs $(x,y)$ where $x \lessdot y$ is a covering relation in $P$, and each element of $P$ occurs in at most one pair in $M$.
If $(x,y) \in M$, then $x \lessdot y$ is a \emph{matched edge} in $M$. If $x \in P$ occurs in a matched edge in $M$, then $x$ is \emph{matched}. Otherwise, $x$ is \emph{unmatched}.  In the case of a simplicial poset $P$, we require that a matching leaves the minimal element unmatched.
\end{defn}

\begin{defn}
Let $P$ be a ranked poset, and let $M$ be a matching on $P$. Consider the Hasse diagram of $P$ as a directed graph, with an edge $v \rightarrow u$ if $v \gtrdot u$.  Reverse the direction of each edge in the Hasse diagram which corresponds to a matched edge in $M$.  Let $P(M)$ be the resulting directed graph.  The matching $M$ is \emph{acyclic} if there are no directed cycles in $P(M)$.  In the Hasse diagram for $P$, those edges whose directions have been changed will be said to point \emph{up}, while unchanged edges point \emph{down}.
\end{defn}

Discrete Morse theory allows one to reduce regular cell complexes without changing the homotopy type of their geometric realizations. Roughly speaking, when $X$ is a regular cell complex with face poset $P$, and $M$ is an acyclic matching on $P$, one can collapse cells along matched edges in $M$ without changing the homotopy type of $|X|$. We will apply this method to regular cell complexes associated to simplicial posets. 

The convention that the minimal element of a simplicial poset be left unmatched circumvents a minor technical issue, as the minimal element in a simplicial poset is represented by the empty cell in the associated cell complex and thus plays no role in the geometric realization. Indeed the minimal element will henceforth be ignored. This convention is also taken in \cite{forman-complexes, forman-user}, and it should bring to mind Remark~\ref{rem:rank -1}.

The particular result which we will use to analyze the boolean complex $\Delta(G)$ is stated below.

\begin{thm}[see \cite{forman-complexes, forman-user}] \label{thm:DiscreteMorse}
Let $P$ be a simplicial poset, and let $M$ be an acyclic matching on $P$. For each $r \geq 0$, let $u_r$ denote the number of elements of rank $r$ that are unmatched.  Then the geometric realization of the regular cell complex associated to $P$ is homotopy equivalent to a CW-complex with exactly $u_r$ cells of dimension $r$ for each $r \geq 0$.
\end{thm}

Theorem~\ref{thm:main result} will be proved inductively by constructing an acyclic matching of the simplicial poset $\b(G)$ with all elements of non-negative rank matched except for one element of rank 0 and some number, which we will denote by $\bn(G)$, of maximal elements.  There are two important points to make about such a matching, summarized in the following remark.

\begin{rem} \label{rem:m well defined}
If $M$ is an acyclic matching on a simplicial poset $P$, and the only elements on non-negative rank unmatched by $M$ are one element of rank $0$ and $\bn$ elements of rank $r$, then the geometric realization of the regular cell complex associated to $P$ is homotopy equivalent to $\bn \cdot S^r$. Also, note that the number $\bn$ is determined by the homotopy type of the cell complex $\Delta$ with face poset $P$, and is therefore independent of the matching $M$. In fact, the number $\bn$ is determined by the Euler characteristic of $\Delta$ and the formula 
\[\chi(\Delta) = \chi(\bn \cdot S^r) = 1 + \bn \cdot (-1)^r.\]
\end{rem}

The following lemma is useful for proving the acyclicity of matchings.

\begin{lem} \label{lem:TwoRanks}
Consider a matching $M$ on a ranked poset $P$.  If $M$ has a cycle, then the elements in the cycle lie in two adjacent ranks of $P$.
\end{lem}

\begin{proof}
Because $M$ is a matching, there cannot be two incident upward pointing edges in $P(M)$.  Thus, after moving upward, one must move downward at least once before moving upward again.  So if $x$ is an element in a directed cycle, then no element in the cycle can be more than one rank higher than $x$ in $P$.
\end{proof}

\section{Proof of Theorem~\ref{thm:main result}}\label{section:proof}

In this section we prove Theorem~\ref{thm:main result} by constructing, for each nonempty finite simple graph $G$, an acyclic matching of $\b(G)$ where the only unmatched elements of non-negative rank are one element of rank $0$ and some number of maximal elements. As noted in Remark \ref{rem:m well defined}, this implies that $|\Delta(G)|$ is homotopy equivalent to a wedge of spheres of maximal dimension. Furthermore, the number of such spheres, the boolean number $\bn(G)$, is determined by the Euler characteristic of $\b(G)$, prompting the next definition. 

\begin{defn}
For a nonempty finite simple graph $G$ on $n$ vertices, set
 \[ \bn(G) = (-1)^{n-1} (\chi(\Delta(G))-1). \]
\end{defn}

Although this definition gives a way to calculate the value $\bn(G)$ in Theorem~\ref{thm:main result}, computing the Euler characteristic requires knowing a significant amount about the structure of the poset $\b(G)$, as opposed to the recursive formula in equation~\eqref{eqn:m + edge} which requires only basic graph operations.  

We construct the matchings by induction on the number of edges in $G$. The inductive step is somewhat complicated, and we will in fact produce matchings with more specific properties than are actually needed for the desired conclusion. The inductive hypothesis is stated below, after introducing the following notation.

\begin{defn} \label{defn:include v} For a nonempty finite simple graph $G$ and a vertex $s$ in $G$, let $\b(G)_s \subset \b(G)$ be the subposet of elements containing the letter $s$.
\end{defn}

The goal of this section is to show inductively that the following statement holds for every integer $\e \geq 0$.

\begin{induct}[$\H(\e) $] \label{ind:He}
For every nonempty graph $G$ with at most $\e$ edges, and for every vertex $s$ in $G$, there exists an acyclic matching $M$ on the poset $\b(G)$ with the following properties:
\begin{enumerate}
\item[\h{1}] The only unmatched elements in $\b(G)$ of non-negative rank are one element of rank 0 and $\bn(G)$ maximal elements;
\item[\h{2}] If $G \setminus s$ is nonempty, then, in the restriction of $M$ to the subposet $\b(G)_s$, the only unmatched elements are $\bn(G) + \bn(G \setminus s)$ maximal elements; and 
\item[\h{3}] If $\sigma \lessdot \tau$ is a matched edge in $M$, and $\tau$ contains $s$, then there exist word representatives for $\tau$ and $\sigma$ such that one of the following conditions holds:
   \begin{enumerate} 
   \item $\tau = \sigma s$, or 
   \item $\sigma$ is obtained from $\tau$ by deleting a letter appearing to the left of $s$.
   \end{enumerate}
\end{enumerate}
\end{induct}

We refer to a matching with properties $\h{1}$--$\h{3}$ as an \emph{$\H$-matching of $\b(G)$ at $s$},
or just as an \emph{$\H$-matching}. Notice that the condition $(G \setminus s) \neq \emptyset$ in $\h{2}$
implies that $s$ is not maximal in $\b(G)_s$ and must therefore be matched. When $G$ is a graph with a
single vertex $s$, an $\H$-matching of $\b(G)$ is obtained trivially. Property $\h{2}$ can be extended to
this trivial $\H$-matching if we formally define $\bn(\emptyset) = 1$.

For a nonempty finite simple graph $G$, the maximal elements of $\b(G)$ are products of $|G|$ letters.  Therefore, if $G$ has an $\H$-matching, then Theorem \ref{thm:DiscreteMorse} and $\h{1}$ imply that $\Delta(G) \simeq \bn(G) \cdot S^{|G|-1}$. Property $\h{2}$ is needed to preserve $\h{1}$ for the inductive step, and also to prove the recursive formula in Theorem \ref{thm:main result}. Property $\h{3}$ is needed purely for the purposes of the induction, specifically, to prove acyclicity of the constructed matching.

The next lemma establishes the base case of the induction.

\begin{lem}[Base case] \label{lem:BaseCase}
$\H(0)$ holds, and $\bn(\delta_n) = 0$ for all $n \geq 1$.
\end{lem}
\begin{proof}
We show that for each vertex $s \in \delta_n$, there exists an $\H$-matching at $s$.
Let the vertices of $\delta_n$ be labeled $1, 2, \ldots, n$.  Without loss of generality, suppose that the vertex $s$ is $n$.  Let $M$ be the matching consisting of all covering relations of the form $\sigma \lessdot 1\sigma$ for $\sigma \subseteq \{2, \ldots, n\}$.  First notice that the only unmatched element of non-negative rank in $\b(\delta_n)$ is the vertex $1$, so $M$ has property $\h{1}$, with $\bn(\delta_n) = 0$.  Secondly, there are no unmatched elements in $\b(\delta_n)_s$, unless $n = 1$, proving $\h{2}$.  Property $\h{3}$ follows from the fact that the Coxeter group with graph $\delta_n$ is commutative, so the letters in $\sigma$ and $\tau$ can be permuted at will.

It remains to show that $M$ is acyclic.  This is straightforward, due to the fact that all the matched edges represent adding or removing the letter $1$ from an element in the poset.  Recall from Lemma \ref{lem:TwoRanks} that a cycle is contained in two adjacent ranks of $\b(\delta_n)$, so an ``up'' edge $\sigma \lessdot 1\sigma$ would have to be followed by a ``down'' edge $1\sigma \gtrdot \tau$. The next step must be an up edge $\tau \lessdot 1\tau$, so $\tau$ cannot contain $1$.  However, this implies that $\sigma = \tau$, contradicting the directions of the edges.  Therefore $M$ is acyclic.
\end{proof}

The remainder of the section is devoted to proving the inductive step. Henceforth we assume that $\H(\e)$ holds for some $\e \geq 0$ and consider a graph $G$ with $\e+1$ edges. We show that for an arbitrary vertex $s \in G$, there exists an $\H$-matching of $\b(G)$ at $s$. If $s$ is the endpoint of an edge $e$, then we construct in Lemma~\ref{lem:Edge} the $\H$-matching of $\b(G)$ from $\H$-matchings of $\b(G-e)$ and $\b(G/e)$, which exist by the induction hypothesis. The case when $s$ is an isolated vertex is treated in Lemma \ref{lem:vIsolated} by constructing the required matching from an $\H$-matching of $\b(G\setminus s)$.

\subsection{Considering a non-isolated vertex}\ 

Suppose that $s$ is the endpoint of an edge $e = \{s,t\}$, and set $H := G -e$. By the induction hypothesis, there exists an $\H$-matching $M_H$ of $\b(H)$ at $s$. We want to use this matching to produce an $\H$-matching of $\b(G)$ at $s$. To this end, we first compare the complexes $\b(G)$ and $\b(H)$. The complex $\b(H)$ is obtained from $\b(G)$ by identifying all elements that can be written as $\alpha st \gamma$ with the elements represented by $\alpha ts \gamma$, respectively.  Consequently, there is a canonical projection of posets
\[ \pi \colon \b(G) \twoheadrightarrow \b(H). \]
Now, let $\b(G)_{e} \subset \b(G)$ be the subset of elements that can be written in the form $\alpha st \gamma$; that is, $\b(G)_{e}$ consists of those elements in which both $s$ and $t$ appear, and where it is possible to write $s$ immediately to the left of $t$.  Let $\b(G)_{\widetilde{e}}$ be the set complement of $\b(G)_{e}$ in $\b(G)$. This gives a decomposition of sets
\begin{equation} \label{eqn:BG break down}
 \b(G) =  \b(G)_{\widetilde{e}} \sqcup \b(G)_{e},
\end{equation}
\noindent where $\sqcup$ denotes the disjoint union.  However, this is not a decomposition of posets, as there are covering relations between $\b(G)_{\widetilde{e}}$ and $\b(G)_{e}$, as recorded in the following lemma.  The proof of the lemma is not difficult, and is left to the reader.

\begin{lem} \label{lem:BadEdges}
\ \begin{enumerate}
\item If $\sigma \lessdot \tau$ is a covering relation in $\b(G)$ such that $\tau \in \b(G)_{e}$ and $\sigma \in \b(G)_{\widetilde{e}}$, then $\sigma$ is obtained from $\tau$ by deleting either $s$ or $t$.
\item If $\sigma \lessdot \tau$ is a covering relation in $\b(G)$ such that $\tau \in \b(G)_{\widetilde{e}}$ and $\sigma \in \b(G)_{e}$, then $\tau$ can be written as $\alpha s w_1 w_2 \cdots w_r t \gamma$, where $\sigma$ is obtained from $\tau$ by removing a letter $w_i$ between $s$ and $t$, and the vertices $s, w_1, w_2, \ldots, w_r, t$ form a path in the graph $H$.
\end{enumerate}
\end{lem}

It is easy to see that $\pi$ restricts to a bijective order-preserving map 
 \[\pi_{\widetilde{e}} \colon \b(G)_{\widetilde{e}} \longrightarrow \b(H). \]
However this is not a bijection of posets, as $\b(H)$ can have more covering relations than $\b(G)_{\widetilde{e}}$, corresponding to elements of $\b(G)_{\widetilde{e}}$ covering elements of $\b(G)_{e}$.  The following lemma nevertheless allows us to pull back the $\H$-matching $M_H$ of $\b(H)$ at $s$ to a matching on $\b(G)_{\widetilde{e}}$ along the map $\pi_{\widetilde{e}}$. 

\begin{lem} \label{lem:NoBadEdgesInM}
If $\sigma$ and $\tau$ are elements in $\b(G)_{\widetilde{e}}$ such that $ \pi(\sigma) \lessdot \pi(\tau)$ is a matched edge in $M_H$, then $\sigma \lessdot \tau$ is a covering relation in $\b(G)_{\widetilde{e}}$.
\end{lem}

\begin{proof}
By Lemma \ref{lem:BadEdges}, if $\pi(\sigma) \lessdot \pi(\tau)$ is a covering relation in $\b(H)$ but $\sigma \lessdot \tau$ is not a covering relation in $\b(G)_{\widetilde{e}}$, then $\tau$ can be written in the form $\alpha s w_1 w_2 \cdots w_r t \gamma$, where the vertices $s, w_1, w_2, \ldots, w_r, t$ form a path in the graph $H$, and $\sigma$ is obtained by deleting a letter $w_i$ between $s$ and $t$. Property $\h{3}$ prohibits the covering relation $\pi(\sigma) \lessdot \pi(\tau)$ from being a matched edge in this case.
\end{proof}

Lemma~\ref{lem:NoBadEdgesInM} allows us to define a matching $M_{\widetilde{e}}$ on $\b(G)_{\widetilde{e}}$ by declaring $\sigma \lessdot \tau$ to be a matched edge in $M_{\widetilde{e}}$ if and only if $\pi(\sigma) \lessdot \pi(\tau)$ is a matched edge in $M_H$.  Property $\h{1}$ for $M_H$ implies that the elements of non-negative rank in $\b(G)_{\widetilde{e}}$ that are unmatched in $M_{\widetilde{e}}$ are exactly one element of rank $0$ and $\bn(H)$ maximal elements.  This matching $M_{\widetilde{e}}$ constitutes a part of our $\H$-matching of $\b(G)$ at $s$, and what remains is to produce matched edges among the elements in $\b(G)_{e}$.

Let $F$ denote the graph $G / e$ with $x \in F$ denoting the new vertex representing this edge. Then $F$ has at most $\e$ edges, so there is an $\H$-matching $M_F$ of $\b(F)$ at $x$ by the induction hypothesis. The following lemma allows us to pull the restriction of $M_F$ to $\b(F)_x$ back to a matching on $\b(G)_{e}$.

\begin{lem} There is an isomorphism of posets 
\[\phi \colon \b(G)_{e} \to \b(F)_x \]
defined by 
\[ \phi(\alpha s t \gamma) = \alpha x \gamma. \]
\end{lem}

The proof mainly consists of checking that $\phi$ is a well-defined, order-preserving map. The same reasoning then gives an inverse to $\phi$. This is left to the reader.

We now define a matching $M_{e}$ on $\b(G)_{e}$ by declaring $\sigma \lessdot \tau$ to be a matched edge in $M_{e}$ if and only if $\phi(\sigma) \lessdot \phi(\tau)$ is a matched edge in $M_F$. The following example illustrates this procedure for the graph considered in Example \ref{ex:a2}.

\begin{example} \label{ex:a2Matching}
Let $G$ be the graph $A_2$ with two vertices, $s$ and $t$, and one edge $e$ between them. Then $H = G -e$ is the graph with two disconnected vertices
$s$ and $t$, and $F = G/e$ is the graph on a single vertex $x$. The picture below shows a Hasse diagram for the poset $\b(G)$, excluding the empty
word, illustrating how $\b(G)$ decomposes into parts $\b(G)_{\widetilde{e}}$ and $\b(G)_{e}$, and how these parts can be related to $\b(H)$ and 
$\b(F)_x$, respectively.
\[
\xymatrix{
  {\overset{ts}{\bullet}} \ar@*{[|(5)]}@{-}[]+<0pt,-3.5pt>;[d]+<0pt,3pt> \ar@{-}[]+<0pt,-3.5pt>;[dr]+<0pt,3.5pt>& &   &  {\overset{ts}{\bullet}} \ar@*{[|(5)]}@{-}[]+<0pt,-3.5pt>;[d]+<0pt,3pt> \ar@{-}[]+<0pt,-3.5pt>;[dr]+<0pt,3.5pt> &  &      {\overset{st}{\bullet}}\ar@{-}[]+<0pt,-3.5pt>;[dll]+<0pt,3.5pt> \ar@{-}[]+<0pt,-3.5pt>;[dl]+<0pt,3.5pt> & &  {\overset{x}{\bullet}} \\
    {\underset{s}{\bullet}}& {\underset{t}{\bullet}}& & {\underset{s}{\bullet}}   & {\underset{t}{\bullet}} && \\
    && 
    \save "1,4"."2,5"="Box1"*[F-:<10pt>]\frm{} \restore
    \save "Box1"+<0.65cm,-2.2cm>*{\tiny{\b(G)_{\widetilde{e}}}} \restore
    \save "1,1"."2,2"="Box2"*[F-:<10pt>]\frm{} \restore
    \save "Box2"+<0.65cm,-2.2cm>*{\tiny{\b(H)}} \restore
    \save "1,6"."1,6"="Box3"*[F-:<10pt>]\frm{} \restore
    \save "Box3"+<0.25cm,-.7cm>*{\b(G)_{e}} \restore
    \save "1,8"."1,8"="Box4"*[F-:<10pt>]\frm{} \restore
    \save "Box4"+<0.25cm,-.7cm>*{\tiny{\b(F)_x}} \restore
    \ar@{<~}"Box2"+<1.7cm,-0.6cm>;"Box1"+<-0.4cm,-0.6cm>^{\pi_{\widetilde{e}}}
        \ar@{<~>}"Box3"+<0.4cm,0cm>;"Box4"+<-0.3cm,0cm>^{\phi}
    \save "1,4"+<-0.8cm,0.8cm>."2,6"="BigBox"*++++++[F-:<10pt>]\frm{} \restore
    \save "BigBox"+<2.5cm,-4.0cm>*{\b(G)} \restore
}
\]
The poset $\b(H)$ has an $\H$-matching at $s$, indicated on the Hasse diagram by bold lines. 
The poset $\b(F)$ has a trivial $\H$-matching at $x$, whose restriction to $\b(F)_x = \b(F)$ 
is also trivial. Pulling back these $\H$-matchings to $\b(G)$ along $\pi_{\widetilde{e}}$ 
and $\phi$, one obtains an $\H$-matching of $\b(G)$ at $s$, indicated by bold lines in the 
Hasse diagram.
\end{example}

Example \ref{ex:a2Matching} is somewhat anomalous in that $\pi_{\widetilde{e}}$ is an isomorphism 
of posets here. In general there may be more coverings in $\b(H)$ than in $\b(G)_{\widetilde{e}}$, 
so the set-theoretic inverse of $\pi_{\widetilde{e}}$ will not be order-preserving. Another anomaly 
is that in this example, the graph $F$ consists of a single vertex, so $\b(F)$ has a trivial
$\H$-matching, in which $x$ is the unique unmatched element. In other cases (when $G \neq A_2$), the
corresponding graph $F$ has two or more vertices, and Property $\h{2}$ implies that the elements 
in $\b(G)_{e}$ that are unmatched in $M_{e}$ are exactly $\bn(F) + \bn(F\setminus x)$ maximal elements.
(In particular $x$ is matched.) This leads to the following lemma.

\begin{lem} \label{lem:Edge}
Assume $G$ is not isomorphic to $A_2$. The matching $M$ on $\b(G)$ obtained by combining the matching $M_{\widetilde{e}}$ on $\b(G)_{\widetilde{e}}$ and the matching $M_{e}$ on $\b(G)_{e}$ is an $\H$-matching of $\b(G)$ at $s$.  Furthermore,
\begin{equation*}
\bn(G) = \bn(G - e) + \bn(G/e) + \bn(G -[e]).
\end{equation*}
\end{lem}

\begin{proof}
It is clear that $M$ is a matching on $\b(G)$. The unmatched elements of non-negative rank in $\b(G)$ are exactly one element of rank 0 and $\bn(G)$ maximal elements, where
\begin{align}
\label{eqn: G in terms of H and F}\bn(G) &= \bn(H) + \bn(F) + \bn(F \setminus x) \\
\nonumber     &= \bn(G - e) + \bn(G/e) + \bn(G -[e]).
\end{align}
This proves $\h{1}$ for $M$, and, indeed, the recursive formula for the function $\bn$.

Now consider the restriction of the matching $M$ to elements of $\b(G)$ which contain the letter $s$.  The set of such elements in $\b(G)_{\widetilde{e}}$ which are unmatched correspond to the unmatched elements in $\b(H)_s$.  By Property $\h{2}$ of the matching $M_H$, these are $\bn(H) + \bn(H \setminus s)$ maximal elements. Every element of $\b(G)_{e}$ contains $s$, so this set of elements still gives exactly $\bn(F) + \bn(F\setminus x)$ unmatched maximal elements. Since $H \setminus s = G\setminus s$, there are exactly $\bn(G) + \bn(G\setminus s)$ unmatched elements in $\b(G)_s$ by equation~\eqref{eqn: G in terms of H and F}, proving $\h{2}$ for $M$.

To prove $\h{3}$ for $M$, recall that every matched edge in $M$ lies in either $M_{\widetilde{e}}$ or $M_{e}$. A matched edge in $M_{\widetilde{e}}$ certainly has the required form, since it comes from a matched edge in $M_H$, which is an $\H$-matching at $s$. A matched edge $\sigma \lessdot \tau$ in $M_{e}$ corresponds to a matched edge $\phi(\sigma) \lessdot \phi(\tau)$ in $M_F$ between two elements containing $x$. But $M_F$ is an $\H$-matching at $x$, and Property $\h{3}$ for $M_F$ implies that $\phi(\sigma)$ and $\phi(\tau)$ can be written so that $\phi(\sigma)$ is obtained from $\phi(\tau)$ by deleting a letter to the left of $x$. It follows that $\tau$ and $\sigma$ can be written so that $\sigma$ is obtained from $\tau$ by deleting a letter to the left of $s$.

It remains to show that $M$ is acyclic. Recall the set decomposition of $\b(G)$ in equation~\eqref{eqn:BG break down}, and the covering relations between $\b(G)_{e}$ and $\b(G)_{\widetilde{e}}$ described in Lemma \ref{lem:BadEdges}.  Additionally, note that in the matching $M$, there are no matched edges between these two sets.

A cycle in $M$ lying entirely in $\b(G)_{\widetilde{e}}$ would induce a cycle in $\b(H)$, contradicting the acyclicity of $M_H$. Similarly, the acyclicity of $M_F$ implies that no cycle lies entirely in $\b(G)_{e}$. It remains only to show that no cycle involves both of these sets.  If there is such a cycle $C$, then the covering relations in $C$ going between the disjoint sets must be ``down'' edges.  To move from $\b(G)_{e}$ to $\b(G)_{\widetilde{e}}$, we must delete either $s$ or $t$ from an element of the form $\alpha st \gamma$. Traversing along $C$, we alternately move up along a matched edge in $M$ by adding an element, and down along an unmatched edge by deleting an element. Property $\h{3}$ guarantees that at no step in this process does the letter $s$ appear to the left of the letter $t$.  Therefore it is impossible to get back to $\b(G)_{e}$ by traversing along $C$, contradicting the assumption that $C$ is a cycle. We deduce that $M$ is acyclic.
\end{proof}

\subsection{Considering an isolated vertex}\ 

We now turn our attention to the case where $s$ is an isolated vertex in $G$.  That is, $s$ has no incident edges in $G$.

\begin{lem} \label{lem:vIsolated}
If $s \in G$ is an isolated vertex, and $\b(G\setminus s)$ has an $\H$-matching at $x$ for some $x \in (G \setminus s)$, then $\b(G)$ has an $\H$-matching at $s$.
\end{lem}

\begin{proof}
Set $H := G\setminus s$. Notice that $H$ is nonempty since we are assuming that $G$ has $\e+1 \geq 1$ edges.
Since $s$ is an isolated vertex in $G$, the letter $s$ commutes with every element in $\b(H)$. 
Hence we have a decomposition of sets 
\begin{equation}\label{eqn:isolated v decomp}
\b(G) = \b(H) \sqcup (\b(H))s \sqcup \{s\}.
\end{equation}
This is not a decomposition of posets as every element $\sigma \in \b(H)$ is covered by $\sigma s \in (\b(H))s$, and $s$ is covered by $ts \in (\b(H))s$ for every vertex $t \in H$. However these are the only covering relations between distinct parts in the decomposition of equation~\eqref{eqn:isolated v decomp}.

Now suppose that $M_H$ is an $\H$-matching of $\b(G \setminus s)$ at some vertex $x$. Let $\sigma_1, \ldots, \sigma_{\bn(H)}$ be the maximal elements in $\b(H)$ which are unmatched by $M_H$, and denote the unmatched rank $0$ element by $1$. We construct a matching $M$ of $\b(G)$ from $M_H$, by including matched edges as follows. 
\begin{itemize}
  \item If $\sigma \lessdot \tau$ is a matched edge in $M_H$, then $\sigma \lessdot \tau$ and $\sigma s \lessdot \tau s$ are matched edges in $M$
  \item $s \lessdot 1s$ is a matched edge in $M$
  \item $\sigma_i \lessdot \sigma_i s$ is a matched edge in $M$ for each $i \in \{1, \ldots, \bn(H)\}$
\end{itemize} 

To see that $M$ is acyclic, first note that no cycle can be contained entirely in $\b(H)$ or $(\b(H))s$ since $M_H$ is acyclic. Next note that no cycle can include the element $s \in \b(G)$, since from $s$ one can only move up to $1s$, from which one can only move to $1$, whence one cannot escape. Therefore any cycle would have to be contained in the blocks $\b(H)$ and $(\b(H))s$, intersecting both. In particular, at some point, the cycle moves from $\b(H)$ to $(\b(H))s$. The only way to do this is by moving along one of the matched edges $\sigma_i \lessdot \sigma_i s$, so a cycle has to pass through some $\sigma_i$. However, there is no way to move down to $\sigma_i$, so it cannot be part of a cycle. We conclude that $M$ is acyclic.

We now show that $M$ satisfies $\h{1}$ -- $\h{3}$. Property $\h{1}$ is immediate, since the only unmatched element in $M$ of non-negative rank is the rank $0$ element $1$. The unmatched elements in $\b(G)_s$ are the maximal elements $\sigma_1s, \ldots, \sigma_{\bn(H)}s$, proving $\h{2}$. Property $\h{3}$ follows from the fact that, as $s$ is an isolated point in $G$, any element in $\b(G)$ containing $s$ can be written so that $s$ appears on the far right.
\end{proof}

\begin{rem}
It is interesting to note a fact that arises in the preceding proof: if the graph $G$ has an isolated vertex, then $\bn(G) = 0$.  In particular, this implies that $\Delta(G)$ is contractible.  This fact will be discussed further in Section~\ref{section:corollaries}.
\end{rem}

\subsection{Completing the proof}\ 

It is now straightforward to complete the proof of Theorem \ref{thm:main result}. 

\begin{proof}[Proof of Theorem~\ref{thm:main result}] By Theorem \ref{thm:DiscreteMorse}, the first claim follows if we can show that $\H(\e)$ holds for all $\e \geq 0$. We prove this by induction on $\e$. The base case, where $\e = 0$ is Lemma \ref{lem:BaseCase}. For the inductive step, assume that $\H(\e)$ holds for some $\e \geq 0$ and consider a graph $G$ with $\e+1$ edges. Let $s$ be a vertex in $G$. If $s$ is the endpoint of an edge $e$ in $G$, then $\b(G)$ has an $\H$-matching at $s$ by Lemma \ref{lem:Edge} when $G$ is not isomorphic to $A_2$, and by Example \ref{ex:a2Matching} in the special case $G \cong A_2$. If $s$ is isolated, then, since $\e+1 \geq 1$, there is some vertex $s' \in G$ that is the endpoint of an edge $\{s',t\}$. Since the graph $G \setminus s$ also has $\e+1$ edges, Lemma \ref{lem:Edge} implies that $\b(G \setminus s)$ has an $\H$-matching at $s'$. Lemma \ref{lem:vIsolated} then implies that $\b(G)$ has an $\H$-matching at $s$.

The inductive formula for the boolean number $\bn(G)$ follows from Lemma~\ref{lem:Edge}. The initial conditions $\bn(\delta_n) = 0$ and $\bn(A_2) = 1$ are consequences of Lemma~\ref{lem:BaseCase} and Example \ref{ex:a2Matching} (or Example \ref{ex:a2}), respectively.
\end{proof}

\begin{rem}
An alternative proof of Theorem~\ref{thm:main result} proceeds as follows.  First, one proves that the complex $\Delta(G)$ is shellable.  To obtain the shelling, we represent each rank $|G|-1$ element of $\b(G)$ by its lexicographically least word representative, and order these least representatives lexicographically.  Shellability implies that the homotopy type of the complex is a wedge of spheres.  To calculate the number $\bn(G)$ of such spheres, one can adapt the inductive argument presented in this section.  The shellability result has been obtained independently by Jonsson and Welker in \cite{jonsson}, and we are grateful to Vic Reiner for pointing out this reference to us.
\end{rem}

\section{Corollaries to Theorem~\ref{thm:main result}}\label{section:corollaries}

Theorem~\ref{thm:main result} gives the homotopy type of the boolean complex for any finite simple graph $G$, specifically that it is the wedge of $\bn(G)$ spheres $S^{|G|-1}$.  For some classes of graphs, we can obtain more specific data regarding the function $\bn$.  The most interesting of these results are listed below, with some proofs left to the reader.

In the case where the graph $G$ has any leaves (vertices of degree $1$), the recursion of equation~\eqref{eqn:m + edge} can be simplified by Proposition~\ref{prop:disjoint union}.

\begin{cor}\label{cor:tree recursion}
If $G$ has any leaves, then the computation of $\bn(G)$ can be simplified as depicted in Figure~\ref{fig:inductive counting}.
\end{cor}

\begin{figure}[htbp]
\begin{eqnarray*}
\bn\left(\parbox[c]{.66in}{\epsfig{file=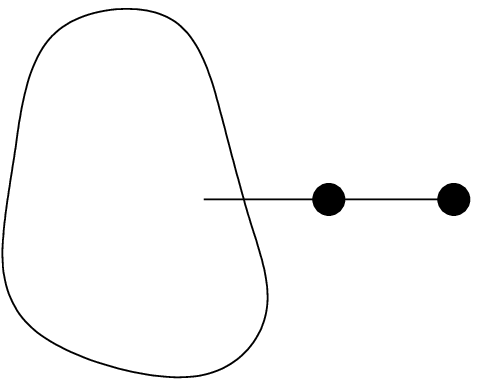,scale=.35}} \right) &=& \bn\left(\parbox[c]{.48in}{\epsfig{file=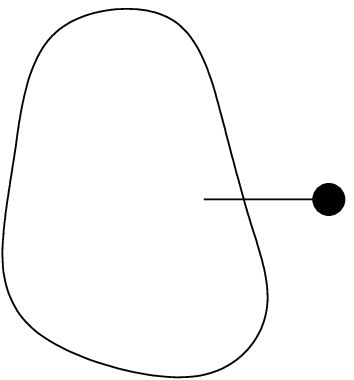,scale=.35}} \right) + \bn\left(\parbox[c]{.39in}{\epsfig{file=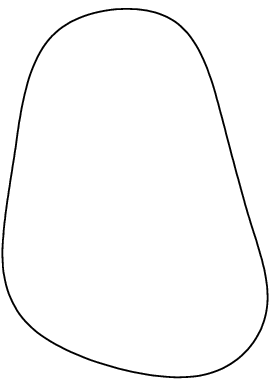,scale=.35}} \right)\\
\rule[0mm]{0mm}{10mm}
\bn\left(\parbox[c]{.48in}{\epsfig{file=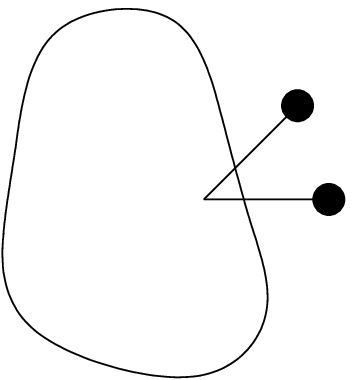,scale=.35}} \right) &=& \bn\left(\parbox[c]{.48in}{\epsfig{file=tree-pic3.eps,scale=.35}} \right)
\end{eqnarray*}
\caption{Simplification for recursively calculating $\bn(G)$ when the graph $G$ has a leaf.  The region without detail is the rest of $G$.}\label{fig:inductive counting}
\end{figure}

The recursion in Corollary~\ref{cor:tree recursion}, together with the initial conditions of Theorem~\ref{thm:main result}, enables efficient calculation of the boolean number of any tree.  Moreover, it shows that if $T$ is a tree with more than one vertex, then $\bn(T) > 0$.  
The recursive formula for calculating $\bn$ shows that adding edges does not decrease the
boolean number. From this and Proposition~\ref{prop:disjoint union} and we draw the following conclusion.

\begin{cor}\label{cor:disjoint point}
A finite simple graph $G$ has an isolated vertex if and only if $\bn(G)=0$.  That is, the center of a Coxeter group contains a generator of the group if and only if the group's boolean complex is contractible.
\end{cor}

The matchings constructed in Section \ref{section:proof} actually allow us to determine the homotopy type of the skeleta of the boolean complex.

\begin{cor}\label{cor:skeleton}
Let $G$ be a finite simple graph.  For each $r \in [0, |G|-1]$, the $r$-skeleton of the complex $\Delta(G)$ is homotopy equivalent to a wedge of $r$-spheres.
\end{cor}

\begin{proof}
The matching on $\b(G)$ constructed in Section~\ref{section:proof} restricts to a matching on the face poset of the $r$-skeleton, whose only unmatched elements are one of rank $0$ and some number of rank $r$.
\end{proof}

The number of spheres in Corollary~\ref{cor:skeleton} can be computed from the rank generating function of $\b(G)$, which satisfies a recursion analogous to that of the function $\bn$.  More precisely, if there are $f_{r'}$ elements of rank $r' \ge 0$ in $\b(G)$, and there are $v_{r+1}$ unmatched elements of rank $r+1$ in the restriction of the matching to the $(r+1)$-skeleton of $\Delta(G)$, then there are $f_r - (f_{r+1} - v_{r+1})$ unmatched elements of rank $r$ in the restriction of the matching to the $r$-skeleton.

We now determine the homotopy types of the boolean complexes of some specific families of graphs.

\begin{defn}
For $n \ge 1$, let $S_n$ be the tree on $n$ vertices with a single vertex of degree $n-1$.
\end{defn}

\begin{cor}\label{cor:star}
For all $n \ge 1$, $\bn(S_n) = 1$.  That is, the boolean complex $\Delta(S_n)$ is homotopy equivalent to $S^{n-1}$.
\end{cor}

\begin{defn}
For $n \ge 1$, let $K_n$ be the complete graph on $n$ vertices.
\end{defn}

\begin{cor}\label{cor:complete}
For all $n \ge 1$, the function $\bn(K_n)$ satisfies the recurrence
\begin{equation*}
\bn(K_n) = (n-1)\big(\bn(K_{n-1}) + \bn(K_{n-2}) \big),
\end{equation*}
\noindent with initial values $\bn(K_1) = 0$ and $\bn(K_2) = 1$.
\end{cor}

Observe that the sequence $\{\bn(K_n)\}$ is sequence A000166 of \cite{oeis}, which is the number of derangements of $[n]$.  Reiner and Webb use character theory to obtain this same result, under the guise of the complex of injective words \cite{reiner-webb}.

From equation~\eqref{eqn:m + edge}, we see that the function $\bn$ is monotonically increasing with respect to edge addition.  More precisely, if $H \subseteq G$ is obtained by deleting some edges from the graph $G$, then $\bn(H) \le \bn(G)$.  One could ask when this inequality is strict, and when there is equality.

\begin{cor}
Fix a finite simple graph $G$.  Obtain $H \subseteq G$ by removing an edge of $G$.  Then $\bn(H) = \bn(G)$ if and only if $G$ has an isolated vertex, in which case $\bn(G) = \bn(H) = 0$.
\end{cor}

\begin{proof}
Let $H = G-e$ for some edge $e$ in $G$.  If $G$ has an isolated vertex, then $\bn(G) = 0$, and so $\bn(H)$ must also be $0$ because $\bn(H) \le \bn(G)$.

Suppose now that $\bn(G) = \bn(H)$.  Then by Theorem~\ref{thm:main result}, it must be that $\bn(G/e)$ and $\bn(G-[e])$ are both $0$.  Therefore, each of these graphs contains an isolated vertex.  If $G/e$ contains an isolated vertex, then either $G$ has an isolated vertex, in which case we are done, or the edge $e$ comprises its own connected component in $G$.  In the latter case, then, requiring that $G-[e]$ have an isolated vertex forces $G$ to have an isolated vertex as well.
\end{proof} 

Given the recursive formula for calculating the boolean number, it is natural to compare $\bn$ to the 
\emph{universal edge elimination polynomial} $\xi(G,x,y,z)$, introduced in \cite{agm}. In \cite{CKRT} it is shown that these graph invariants are related by
 \begin{equation} \label{eq:xi} 
    \bn(G) = (-1)^{|G|}\xi(G,0,-1,1). 
 \end{equation}
The polynomial $\xi$ is defined on all graphs (simple or not) and has the universal property that any
graph polynomial satisfying an edge recursion involving deletion, extraction and contraction is an
evaluation of $\xi$. However, there is a subtlety involved in relating $\bn$ to $\xi$: $\bn$ is only
defined on \emph{simple} graphs, and its recursive formula involves deletion, extraction and
\emph{simple} contraction, in which loops and multiple edges that appear upon contracting an edge are
deleted. This subtlety can be overcome to obtain equation~\eqref{eq:xi}, and the details 
are carried out in \cite{CKRT}.

\begin{cor} \label{}
Let $G$ be a nonempty finite simple graph with vertex set $V$ and edge set $E$. Then
 \[ \bn(G) = \sum_{\substack{B \subseteq E \\ V(B) = V }} (-1)^{|G| + |B|-\kappa(B)}, \]
where $V(B)$ is the set of vertices that are incident to the edges in $B$, and $\kappa(B)$ is the number
of components in the graph with vertex set $V$ and edge set $B$. 

In particular, when $T$ is a nonempty finite tree, $\bn(T)$ is the number of spanning forests in $T$.
\end{cor}
\begin{proof}
The first claim follows from the formula given for $\xi$ in \cite{agm} and equation~\eqref{eq:xi}. For the second claim, notice that for a tree with vertex set $V$ and edge set $E$, the condition $V(B) = V$ on a subset $B \subseteq E$ is equivalent to saying that $(V,B)$ is a spanning forest. Next observe 
that $\kappa(B) = |T| - |B|$, so 
 \[ (-1)^{|T| + |B| - \kappa(B)} = (-1)^{2|B|} = 1, \]
and the sum simply counts the number of spanning forests.
\end{proof}

\section{Homology of the boolean complex} \label{sec:Homology}

At this point we know that the boolean complex of a Coxeter system has the homotopy type of a wedge of spheres of maximal dimension. We also have an expedient way of calculating the number of spheres recursively from the Coxeter graph $G$. However we have little knowledge of the geometric or combinatorial relevance of the spheres in the wedge sum. Our proof of Theorem \ref{thm:main result} used discrete Morse theory to determine the homotopy type of the boolean complex by producing an acyclic matching on its face poset, the boolean ideal. This matching leaves one element of rank $0$ and $\bn(G)$ maximal elements unmatched. The unmatched element of rank $0$ represents a point $x_0 \in |\Delta(G)|$, and discrete Morse theory allows us to collapse all cells corresponding to matched cells down to this point without changing the homotopy type of $|\Delta(G)|$, thus obtaining the homotopy equivalence $|\Delta(G)| \simeq \bn(G) \cdot S^{|G|-1}$, where each unmatched maximal element corresponds to a sphere summand. However, this process involves several choices, and the unmatched cells are not at all canonical representatives for the spheres. Furthermore, the cell corresponding to an unmatched element is generally not the boundary of a sphere in $|\Delta(G)|$, and so does not represent the sphere in any geometric sense.

This discrepancy can by rectified by considering the homology of the boolean complex and we shall do so in this section. To avoid issues of orientation, we take coefficients in $\Z/2$. Determining the homology groups is a simple matter, since the homology of $\Delta(G)$ is the reduced homology of $|\Delta(G)|$ (reduced because of the rank $-1$ element of $\Delta(G)$).  Thus, by Theorem \ref{thm:main result},
\[ H_*(\Delta(G);\Z/2) \cong 
\begin{cases} 
(\Z/2)^{\bn(G)}, &\text{if $* = |G|-1$;}\\
0, &\text{if $* \neq |G|-1$.}
\end{cases}
\]
The important thing is therefore not calculating the homology, but finding the generating cycles, as a cycle representing a generator of $H_{|G|-1}(\Delta(G);\Z/2)$ is a canonical representative for the corresponding sphere summand. Indeed, such a cycle consists of a sum of cells, which (together with their boundaries) form a sphere in the geometric realization.

The authors have calculated explicit generating cycles for boolean complexes of Coxeter systems of type $A_n$, for $1 \leq n \leq 6$ and the results are presented below. The results also apply to Coxeter systems of type $B_n$ as they have the same unlabeled Coxeter graphs. We omit the calculations and only present the results.

To establish notation, we consider the Coxeter graph of $A_n$ to be a graph with vertices labeled $1$ through $n$ and an edge going from $k$ to $k+1$ for $1 \leq k <n$. The cycles generating the homology of $\Delta(A_n)$ are each a sum of cells of maximum dimension. We shall denote each cell by a string representing the corresponding maximal element in $\b(A_n)$, enclosed in square brackets.

\begin{example} The cycles generating $H_*(\Delta(A_n);\Z/2)$ for $1 \leq n \leq 6$ are as follows.\\
$H_*(\Delta(A_1);\Z_2) = 0$.\\
$H_1(\Delta(A_2);\Z_2) \cong \Z/2$ with generator $[12] + [21]$.\\
$H_2(\Delta(A_3);\Z_2) \cong \Z/2$ with generator 
\[y:= [123] + [213] + [312] + [321]. \]
$H_3(\Delta(A_4);\Z_2) \cong (\Z/2)^2$ with generators 
\begin{eqnarray*}
y_0 &:=& [1234] + [2134] + [4123] + [4213],\\
y_1 &:=& [3124] + [3214] + [4312] + [4321].
\end{eqnarray*}
$H_4(\Delta(A_5);\Z_2) \cong (\Z/2)^3$ with generators $y_{00}+y_{10}$, $y_{01}+y_{11}$, and $y_{00}+y_{01}$, where
\begin{eqnarray*}
y_{00} &:=& [12345] + [21345] + [51234] + [52134],\\
y_{01} &:=& [41235] + [42135] + [54123] + [54213],\\
y_{10} &:=& [31245] + [32145] + [53124] + [53214],\\
y_{11} &:=& [43125] + [43215] + [54312] + [54321].
\end{eqnarray*} 
$H_5(\Delta(A_6);\Z_2) \cong (\Z/2)^5$ with generators $y_{000}+y_{010}$, $y_{001}+y_{011}$, $y_{100}+y_{110}$, $y_{101}+ y_{111}$, and $y_{000} +y_{001} + y_{100}+y_{101}$, where
\begin{eqnarray*}
y_{000} &:=& [123456] + [213456] + [612345] + [621345],\\
y_{001} &:=& [512346] + [521346] + [651234] + [652134],\\ 
y_{010} &:=& [412356] + [421356] + [641235] + [642135],\\
y_{011} &:=& [541236] + [542136] + [654123] + [654213],\\
y_{100} &:=& [312456] + [321456] + [631245] + [632145],\\
y_{101} &:=& [531246] + [532146] + [653124] + [653214],\\
y_{110} &:=& [431256] + [432156] + [643125] + [643215],\\
y_{111} &:=& [543126] + [543216] + [654312] + [654321].
\end{eqnarray*}
\end{example}

There are clear patterns in these results, and we have tried to suggest this by the notation. The element $y_0 \in \b(A_4)$ is formed by taking the first two terms of the element $y \in \b(A_3)$ and adding the letter $4$ on either side, while $y_1$ is formed similarly from the last two terms of $y$. This rule continues, so for example $y_{101}$ is obtained by adding the letter $6$ on either side of the last two terms of $y_{10}$. From these calculations, it seems that the elements formed in this way play a key role in enumerating the homology classes of the boolean complexes. It would be very interesting to determine the exact rules dictating how cells group together to form homology classes, for type $A_n$ as well as other graphs.

It is also interesting to note that every cell of maximal dimension occurs in some homology class.

\section{Further questions}\label{section:follow-up}

Several questions about boolean complexes and boolean numbers are suggested by this work and remain unanswered.  These range from topological questions to more combinatorial ones.  We highlight a selection of these problems here.

\begin{question}
What are the cycles representing the homology of the boolean complex of a graph?  Even in the relatively simple case of type $A_n$, where one can enumerate the maximal cells explicitly, there is a rich combinatorial structure waiting to be uncovered. 
\end{question}

\begin{question}
Is there a relationship between the boolean complex of a Coxeter system and its Coxeter complex and/or Tits building? The boolean complex is clearly a ``coarser'' construction, as it depends only on the unlabeled Coxeter graph, while knowing the Coxeter complex or a Tits building is equivalent to knowing the labeled Coxeter graph, yet there are many similarities. For instance, the Coxeter complex of a finite reflection group has the homotopy type of a wedge of spheres.
\end{question}

The boolean number $\bn(G)$ of a graph $G$, is clearly a graph invariant. Letting $\G_n$ be the class of graphs with $n$ vertices, we have a function $\bn_n \colon \G_n \to \Z, G \mapsto \bn(G)$. This prompts the next two questions.

\begin{question}
What can be said of the image of $\bn_n$? For example, the corollaries in Section~\ref{section:corollaries} show that $1$ is in the image, and that the largest value is $\bn(K_n)$.  However, an easy example shows that the image is not equal to the interval $[1, \bn(K_n)]$: there is no graph $G$ on $4$ vertices with $\bn(G)$ equal to $4$.  Is there a way to decide which values are in the image, besides examining all graphs on $n$ vertices?  That is, is there another way to characterize the image of $\bn_n$?
\end{question}

\begin{question}
For graphs $G$ and $H$ on $n$ vertices, can anything special be said if $\bn(G) = \bn(H)$?  Note that there are non-isomorphic graphs with this property: let $G$ be a path of $5$ vertices, and let $H$ be a triangle with two extra vertices adjacent to one corner.  Then $\bn(G) = \bn(H) = 3$.
\end{question}

Recall from Remark~\ref{rem:m well defined} that the function $\bn$ is closely related to the rank generating function of the poset $\b(G)$.  Thus the previous two questions could be rephrased in terms of this function.

\begin{question}
Are there other families of graphs (besides the complete graphs) that give rise to an interesting sequence?  If so, is there a bijection between the homology generators for the boolean complexes of such a family and other interesting objects?
\end{question}

This last question is taken up in \cite{CKRT}, where it is shown where the boolean numbers of complete
bipartite graphs are expressed in terms of the Stirling numbers of the second kind, and
that the boolean numbers of Ferrers graphs of corresponding to staircase shapes are the Genocchi numbers of the second kind.

\section{Acknowledgments}

The authors are grateful to Patricia Hersh and Richard Ehrenborg for helpful discussions at the onset of this work, as well as to John Shareshian for thoughtful feedback and suggestions, and the comments of a dedicated referee.  We also thank Raleigh's in Berkeley, and Stephanie in particular, for providing the perfect setting for working on this project.  Finally, we appreciate the support of MSRI as well as their sponsorship of the programs during which much of this work was done.


\begin{thebibliography}{99}

\bibitem{agm} I.~Averbouch, B.~Godlin, and J.~A.~Makowski, An extension of the bivariate chromatic polynomial, preprint.

\bibitem{bjorner} A.~Bj\"{o}rner, Posets, regular CW complexes and Bruhat order, \textit{European J.~Combin.} \textbf{5} (1984), 7--16.

\bibitem{bjornerbrenti} A.~Bj\"{o}rner and F.~Brenti, \textit{Combinatorics of 
Coxeter Groups}, Graduate Texts in Mathematics 231, Springer, New York, 2005.

\bibitem{bjornerwachs} A.~Bj\"{o}rner and M.~Wachs, On lexicographically shellable posets, \textit{Trans.~Amer.~Math.~Soc.} \textbf{277} (1983), 323--341.

\bibitem{brenti} F.~Brenti, A combinatorial formula for Kazhdan-Lusztig polynomials, \textit{Invent.~Math.} \textbf{118} (1994), 371--394.

\bibitem{CKRT} A.~Claesson, S.~Kitaev, K.~Ragnarsson, and B.~E.~Tenner, Boolean complexes for Ferrers graphs, {\tt arXiv:0808.2307}.

\bibitem{ehrenborg} R.~Ehrenborg and G.~Hetyei, The topology of the independence complex, \textit{Europ.~J.~Combin.} \textbf{27} (2006), 906--923.

\bibitem{farmer} F.~D.~Farmer, Cellular homology for posets, \textit{Math.~Japon.} \textbf{23} (1978/79), 607--613.

\bibitem{forman-complexes} R.~Forman, Morse theory for cell complexes, \textit{Adv.~Math.} \textbf{134} (1998), 90--145.

\bibitem{forman-user} R.~Forman, A user's guide to discrete Morse theory, \textit{S\'{e}minaire Lotharingien de Combinatoire} \textbf{48} (2002), B48c.

\bibitem{hatcher} A.~Hatcher, \textit{Algebraic Topology}, Cambridge University Press, Cambridge, 2002.

\bibitem{kozlov} D.~N.~Kozlov, Complexes of directed trees, \textit{J.~Combin.~Theory, Ser.~A} \textbf{88} (1999), 112--122.

\bibitem{jonsson} J.~Jonsson and V.~Welker, Complexes of injective words and their commutation classes, {\tt arXiv:0712.2143}.

\bibitem{reiner-webb} V.~Reiner and P.~Webb, The combinatorics of the bar resolution in group cohomology, \textit{J.~Pure Appl.~Algebra} \textbf{190} (2004), 291--327.

\bibitem{oeis}  N.~J.~A.~Sloane, The on-line encyclopedia of integer sequences, published electronically at \hfill \phantom{*} {\tt http:/$\!\!$/www.research.att.com/\~{}njas/sequences/}.

\bibitem{ec1} R.~P.~Stanley, \textit{Enumerative Combinatorics, vol.~1}, Cambridge Studies in Advanced Mathematics, no.~49, Cambridge University Press, Cambridge, 1997

\bibitem{tenner} B.~E.~Tenner, Pattern avoidance and the Bruhat order, \textit{J.~Combin.~Theory, Ser.~A} \textbf{114} (2007), 888--905.
 
\end{thebibliography}
\end{document}